%% file: AlmComplArc_v3.tex
\documentclass[12pt]{article}

\usepackage{amsmath}
\usepackage{amssymb}
\usepackage{amsthm}
\usepackage{latexsym}
\usepackage{cite}
\usepackage{psfrag}
\usepackage{epsfig}
\usepackage{graphicx}
\usepackage[latin1]{inputenc}

\newtheorem{theorem}{Theorem}[section]
\newtheorem{lemma}[theorem]{Lemma}
\newtheorem{corollary}[theorem]{Corollary}
\newtheorem{proposition}[theorem]{Proposition}
\newtheorem{conjecture}[theorem]{Conjecture}
\newtheorem{observation}[theorem]{Observation}
\theoremstyle{definition}
\newtheorem{definition}[theorem]{Definition}
\newtheorem{remark}[theorem]{Remark}

\makeatletter \@addtoreset{equation}{section} \makeatother

\def\B{\mathcal{B}}
\def\C{\mathcal{C}}

\def\K{\mathcal{K}}
\def\M{\mathcal{M}}
\def\N{\mathcal{N}}
\def\O{\mathcal{O}}
\def\S{\mathcal{S}}
\def\U{\mathcal{U}}
\def\F{\mathbb{F}}
\def\PG{\mathrm{PG}}
\def\AG{\mathrm{AG}}

\begin{document}
\title{On Almost Complete Subsets of a Conic in $\mathrm{PG}(2,q)$,
Completeness of Normal Rational Curves and Extendability of Reed-Solomon Codes
\date{}}
\maketitle
\begin{center}
\textbf{D. Bartoli}$^{\text{a}}$\footnote{The research of  D. Bartoli, S. Marcugini, and F.~Pambianco was
 supported in part by Ministry for Education, University
and Research of Italy (MIUR) (Project ``Geometrie di Galois e
strutture di incidenza'')
 and by the Italian National Group for Algebraic and Geometric Structures and their Applications
 (GNSAGA - INDAM).}, \textbf{A. A. Davydov}$^{\text{b}}$\footnote{
 The research of A.A.~Davydov was carried out at the IITP RAS at the expense of the Russian
Foundation for Sciences (project 14-50-00150).}, \textbf{S. Marcugini}$^{\text{a1}}$, and \textbf{F. Pambianco}$^{\text{a1}}$\\
$^{\text{a}}$ \emph{Department of Mathematics and Computer Sciences,}\\
   \emph{Universit{\`a} degli Studi di Perugia}\\
  daniele.bartoli@unipg.it\quad stefano.marcugini@unipg.it\quad fernanda.pambianco@unipg.it\\
$^{\text{b}}$ \emph{Kharkevich Institute for Information Transmission Problems} \\
\emph{ Russian Academy of  Sciences, Moscow,  Russia}\\ adav@iitp.ru\\
\end{center}
\begin{abstract}
A subset $\S$ of a conic $\C$ in the projective plane $\PG(2,q)$ is called \emph{almost complete} (AC-subset for short) if
it can be extended to a larger arc in $\PG(2,q)$ only by the points of $\C\setminus \S$ and by the nucleus of $\C$ when $q$
is even. New upper bounds on the smallest size $t(q)$ of an AC-subset are obtained, in particular,
\begin{align*}
&t(q)<\sqrt{q(3\ln q+\ln\ln q +\ln3)}+\sqrt{\frac{q}{3\ln q}}+4\thicksim\sqrt{3q\ln q};\\
&t(q)<1.835\sqrt{q\ln q}.
\end{align*}
 The new bounds are used to increase regions of pairs $(N,q)$ for which it is proved that every normal rational curve in $\PG(N,q)$ is a complete $(q+1)$-arc or, equivalently, that no $[q+1,N+1,q-N+1]_q$  generalized doubly-extended
Reed-Solomon code can be extended to a $[q+2,N+1,q-N+2]_q$ MDS code.
\end{abstract}



\section{Introduction}\label{sec_Intro}

Let $\PG(N,q)$ be the $N$-dimensional projective space over the Galois field $\F_q$ of order $q$. An $n$-arc in $\PG(N,q)$ with $n>N + 1$ is a set of $n$ points such that no $N + 1$ points belong to the same hyperplane of $\PG(N,q)$. An $n$-arc of $\PG(N,q)$ is complete
if it is not contained in an $(n+1)$-arc of $\PG(N,q)$. In $\PG(N,q)$ with $2\le N\le q-2$, a normal rational curve is any $(q+1)$-arc projectively equivalent to the arc
$\{(1,t,t^2,\ldots,t^N):t\in \F_q\}\cup \{(0,\ldots,0,1)\}$.
 For an introduction to projective geometries over finite fields see \cite
{HirsBook,HirsStor-2001,HirsThas-2015}.

 Let an $[n,k,d]_q$ code be a $q$-ary linear code of length $n$, dimension $k$, and minimum distance $d$. If $d=n-k+1$,
  it is a maximum distance separable (MDS) code. The code dual to an $[n,k,n-k+1]_q$ MDS code is an $[n,n-k,k+1]_q$ MDS code.

Points (in the homogeneous coordinates) of an $n$-arc in $\PG(N,q)$ treated as columns define a generator matrix of an $[n,N+1,n-N]_q$ MDS code. If an $n$-arc in $\PG(N,q)$ is complete then the corresponding $[n,N+1,n-N]_q$  MDS code cannot be extended to an $[n+1,N+1,n-N+1]_q$ MDS code.
  For properties of linear MDS codes and their equivalence  to arcs see e.g.  \cite{HirsBook,HirsStor-2001,HirsThas-2015,Ball-book2015,BallBeule-arXiv2016,Chow-arX2015,HirsKorTorBook,Klein-Stor,%
LandSt,MWS,Roth_book,StorThasDM92,StorThasDM94,Thas_1992_MDS}.

The $j$-th column of a generator matrix of a $[q+1,N+1,q-N+1]_q$ generalized doubly-extended
Reed-Solomon (GDRS) code has the form
 $(v_j,v_j\alpha_j,v_j\alpha_j^2,\ldots,v_j\alpha_j^N)^T$, where $j=1,2,\ldots,q$; $\alpha_1,\ldots,\alpha_q$ are distinct elements of $\F_q$;
 $v_1,\ldots,v_q$ are nonzero (not necessarily distinct) elements of $\F_q$. Also, this matrix contains one more column $(0,\ldots,0,v)^T$ with $v\neq 0$.
 The code, dual to a GDRS code, is a GDRS code too.

Points (in the homogeneous coordinates) of a normal rational curve in $\PG(N,q)$ treated as columns define a generator matrix of a $[q+1,N+1,q-N+1]_q$ GDRS code.
Proposition~\ref{prop1_equiv} is well known.
\begin{proposition}\label{prop1_equiv}
   Let $N$ and $q$ be fixed integers with $2\le N\le q-2$. Moreover, let $q$ be a prime power. The following statements are equivalent:\\
$\bullet$ Every normal rational curve in $\PG(N,q)$ is a complete $(q+1)$-arc;\\
$\bullet$ No $[q+1,N+1,q-N+1]_q$  GDRS code can be extended to a $[q+2,N+1,q-N+2]_q$  MDS code.
\end{proposition}
Due to Proposition \ref{prop1_equiv}, all results given below on completeness of normal rational curves can be reformulated in coding theory language for extendability of GDRS codes.

The completeness of  normal rational curves and related problems  are considered in numerous works starting from Segre's paper
  \cite{Segre_1955} of 1955; see for example
 \cite{HirsBook,HirsStor-2001,HirsThas-2015,Ball-book2015,BallBeule-arXiv2016,Chow-arX2015,HirsKorTorBook,Klein-Stor,LandSt,MWS,Roth_book,%
StorThasDM92,StorThasDM94,Thas_1992_MDS,%
Segre_1955,Ball-JEMS2012,BallBeule-DCC2012,KorStSz_1997_SpFilSubNRC,St_1992_ComplNRC,St_Th_GRS}, where surveys and references can be found. In particular, the following conjecture, connected with the famous Segre's three problems, is well known.

\begin{conjecture}\label{conj1} Let $2\le N\le q-2$. Every normal rational curve in $\PG(N,q)$  is a complete $(q+1)$-arc except for the cases $q$ even and $N\in\{2,q-2\}$ when one point can be added to the curve.
\end{conjecture}

 \begin{remark}\label{rem1_qeven}As a comment to Conjecture \ref{conj1} for $q$ even, note the following. If $N=2$, the point which can be added to a normal rational curve  is unique. But if $N=q-2$, there are many points in $\PG(q-2,q)$ which extend a normal rational curve to a $(q+2)$-arc, see \cite[Theorem 3.10]{StorThasDM94} for the geometrical characterization of these points.
 \end{remark}

 \begin{remark}\label{rem1} If $k\geq q$ then an $[n,k,n-k+1]_q$ MDS code has length $n\leq k+1$, see e.g. \cite{MWS,Roth_book}.
    For $2\le N\le q-2$, the well known \emph{MDS conjecture} assumes that an  $[n,N+1,n-N]_q$ MDS code (or equivalently an $n$-arc in $\PG(N,q)$) has length $n\le q+1$ except for the cases $q$ even and $N\in\{2,q-2\}$ when $n\le q+2$. The MDS conjecture considers all MDS codes (or all arcs) whereas Conjecture~\ref{conj1} says only something about normal rational curves (or GDRS codes). If the MDS conjecture holds for some pair $(N,q)$ then Conjecture~\ref{conj1} holds too, but in general the reverse is not true.
 \end{remark}

   For many pairs $(N,q)$ Conjecture \ref{conj1} is proved, see
 \cite{HirsBook,HirsStor-2001,HirsThas-2015,Ball-book2015,BallBeule-arXiv2016,HirsKorTorBook,Klein-Stor,LandSt,MWS,Roth_book,Thas_1992_MDS,%
Segre_1955,Ball-JEMS2012,BallBeule-DCC2012,Chow-arX2015,KorStSz_1997_SpFilSubNRC,St_1992_ComplNRC,St_Th_GRS} and the references therein; but in general, \emph{completeness of normal rational curves is an open problem}. The main known results are given in Table \ref{tab1}, where $p$ and $p_0(h)$ are \emph{prime}. For rows 1--6 of Table~\ref{tab1}, in fact, the MDS conjecture is proved. In \cite{BallBeule-arXiv2016}, see row 7 of Table \ref{tab1}, it is proved that
a subset of size $3(N-1)-6$ of a normal rational curve in $\PG(N,q)$, $q$ odd, cannot be extended to an arc of size $q+2$. This means that $3N-3\le q+1$ (otherwise the curve could not contain a such subset). So, $N\le\frac{q+4}{3}$. The regions of $N$ in rows 10--11 cover the ones in rows 6--8; we included  rows 6--8 in Table 1 as the methods used for them are useful for further research.

\begin{table}[ht]
\caption{ Pairs $(N,q)$ for which it is proved that every normal rational curve in $\PG(N,q)$ is a complete $(q+1)$-arc}
\begin{tabular}
{@{}c|c|c|c@{}}
  \hline
  no.& $q$&$N$  &Reference \\ \hline
  1&$q=p^{2h+1}$, $p\ge3$, $h\ge1$&$q - \frac{1\vphantom{^{H}}}{4}\sqrt{pq}+\frac{29}{16}p-3<N\le q - 3$&\cite[Table 3.4]{HirsStor-2001}\\  \hline
2&$q=p^h$, $p\ge5$&$q - \frac{1\vphantom{^{H}}}{2}\sqrt{q}+ 1<N\le q - 3$&\cite[Table 3.4]{HirsStor-2001}\\  \hline
  3&$q=p^h\ge23^2$; $p\ge3$;&$q - \frac{1\vphantom{^{H}}}{2}\sqrt{q}-1<N\le q - 3$&\cite[Table 3.4]{HirsStor-2001}\\
&$q\neq5^5,3^6$; $h$ even for $q=3$&&  \\\hline
   4&$q=2^h$, $h>2$&$q - \frac{1\vphantom{^{H}}}{2}\sqrt{q}- \frac{11}{4}<N\le q - 5$&\cite[Table 3.4]{HirsStor-2001}\\  \hline
   5&$q=p$ &$2\le N\le p-1$&\cite{Ball-book2015,Chow-arX2015,Ball-JEMS2012,BallBeule-DCC2012}\\\hline
   6&$q=p^2$&$2\le N\le 2\sqrt{q}-3$&\cite{Ball-book2015,Chow-arX2015,Ball-JEMS2012,BallBeule-DCC2012}\\\hline
  7&$q$ odd&$N\le\frac{q+4}{3}$&\cite[Theorem 1.4]{BallBeule-arXiv2016}\\\hline
   8&all $q$ &$3\le N\le q+2-6\sqrt{q\ln q}\vphantom{^{H^H}}$ & \cite[Theorem 3.3]{St_1992_ComplNRC}   \\\hline
   &$q=p^{2h+1}$; $p\ge p_0(h)$; $p_0(h)$ && \\
   9&is the smallest $\widehat{p}$ satisfying& $2\le N\le q-2$ & \cite[Theorem 3.5]{St_1992_ComplNRC} \\
   &$\sqrt{\widehat{p}}>24\sqrt{(2h+1)\ln\widehat{p}}$&&\\
   &$+\frac{29}{4\widehat{p}^{\,\,h-0.5}}-\frac{20}{\widehat{p}^{\,\,h+0.5}}$&&\\\hline
  10&$q$ odd & $2\le N\le q-2- \sqrt{7(q + 1)\ln q}\vphantom{^{H^{H^H}}}$ &\cite[Theorem 9.2]{KorStSz_1997_SpFilSubNRC} \\  \hline
  11&$q$ even & $3\le N\le q-1- \sqrt{7(q + 1)\ln q}\vphantom{^{H^{H^H}}}$ &\cite[Theorem 9.2]{KorStSz_1997_SpFilSubNRC} \\  \hline
\end{tabular}\label{tab1}
\end{table}

For the problem of completeness of normal rational curves we use tools connected with almost complete subsets of a conic in the projective plane $\PG(2,q)$.

 An $n$-arc in $\PG(2,q)$ is a set of $n$
points no three of which are collinear.
A point $P$ of $\PG(2,q)$ is covered by an arc $\K\subset\PG(2,q)$ if $P$ lies on a bisecant of $\K$. Throughout the paper, $\C=\{(1,t,t^2):t\in \F_q\}\cup \{(0,0,1)\}$ is a fixed conic in $\PG(2,q)$.
Any point subset of $\C$ is an arc.
For even $q$, denote by $\O$ the nucleus of $\C$.
Let
\begin{equation*}
    \M_q:=\left\{\begin{array}{ll}
               \PG(2,q)\setminus \C & \text{if }  q\text{ odd}  \\
               \PG(2,q)\setminus (\C\cup\{\O\}) & \text{if } q\text{ even}
             \end{array}.
    \right.
\end{equation*}
\begin{definition}\label{def1} \textbf{(i)}
In $\PG(2,q)$, an \emph{almost complete subset} of the conic $\C$ (\emph{AC-subset}, for short) is a proper subset of  $\C$ covering all the points of $\M_q$.
An $n$-AC-subset is an AC-subset of size~$n$.

\textbf{(ii)} An  AC-subset is \emph{minimal} if it does not contain a smaller AC-subset.
\end{definition}

Note that an AC-subset $\S$ is an arc that can be extended to a larger arc in $\PG(2,q)$ only by the points of $\C\setminus \S$ and by the nucleus $\O$ when $q$ is even. The term ``almost completeness'' was introduced in \cite[p. 94]{KorStSz_1997_SpFilSubNRC} for objects in the affine plane $\AG(2,q)$.

Denote by $t(q)$  \emph{the smallest size of an AC-subset in $\PG(2,q)$.}

In this work we provide new
  \emph{upper bounds} on  $t(q)$. This is an \emph{open problem}. It is addressed, for example, in
\cite{St_1992_ComplNRC,St_Th_GRS,Kovacs}. In \cite{Kovacs}, by probabilistic methods, it is proved that
\begin{align}\label{eq1_Kovacs}
   t(q)<6\sqrt{q\ln q}.
\end{align}

In \cite[Theorem 3.1]{St_1992_ComplNRC}, using the results and approaches of \cite{St_Th_GRS}, the following connection between $t(q)$ and the completeness of  normal rational curves is proved:

\emph{under the condition
\begin{align}\label{eq1_Storm3.1}
3\le N\le q+2-t(q),
\end{align}
\indent every normal rational curve in $\PG(N,q)$ is a complete $(q+1)$-arc.}

The \emph{aims of this paper} are as follows:
 obtain new upper bounds on the smallest size of an AC-subset of a conic in $\PG(2,q)$;
  using the bounds, extend  regions of pairs $(N,q)$ for which it is proved that every normal rational curve in $\PG(N,q)$ is a complete $(q+1)$-arc.

The paper is organized as follows. In Section \ref{sec_results} the main results of this paper are formulated. In Section \ref{sec_alg}, we consider an estimate of  the number of new covered points in one step of a step-by-step algorithm constructing AC-subsets. In Section \ref{sec_upper bound}, implicit and explicit upper bounds on $t(q)$, based on the results of Section~\ref{sec_alg}, are obtained. In Section~\ref{sec_comput}, computer assisted bounds on $t(q)$ are studied. In Section~\ref{ses_NRC}, new  bounds on $t(q)$ are applied to the problem of completeness of normal rational curves. Finally, in Appendix tables of
 the smallest known sizes $\overline{t}(q)$ of AC-subsets in $\PG(2,q)$ are given.
\section{The main results}\label{sec_results}
We introduce the following set of prime powers.
\begin{align}
  &Q_1:=\{8\le q\le 139129,~ q=p^m,~p \text{ prime},~m\ge2\}.\label{eq1_region_Q1}
\end{align}
Throughout the paper we denote
\begin{align}\label{eq1_Phi}
 &\Phi(q)=\sqrt{q(3\ln q+\ln\ln q +\ln3)}+\sqrt{\frac{q}{3\ln q}}+4\thicksim \sqrt{3}\sqrt{q\ln q};\displaybreak[3]\\
 &\Theta(q)=\left\{\begin{array}{ccc}
                     1.62\sqrt{q\ln q} & \text{for} & 8\le q\le 17041 \\
                     1.635\sqrt{q\ln q} & \text{for} & 17041< q\le 33013 \\
                     1.674\sqrt{q\ln q} & \text{for} & q\in Q_1\\
                     \min\{1.835\sqrt{q\ln q},\Phi(q)\}& \text{for}  & \text{all }q\ge5
                   \end{array}
 \right.,\label{eq1_Theta}
\end{align}
where
\begin{align*}
 \min\{1.835\sqrt{q\ln q},\Phi(q)\}=\left\{\begin{array}{ccc}
                     1.835\sqrt{q\ln q}& \text{for}  &  q< 12755807\\
                     \Phi(q)& \text{for}  &12755807\le q
                   \end{array}
 \right..
\end{align*}

The main result of this paper is Theorem \ref{th1_main_bnd} based on Theorems \ref{th3_explicit bndC}, \ref{th3_1835}, and \ref{th4_computBnd}.
\begin{theorem}\label{th1_main_bnd}
 The following upper
bound on the smallest size $t(q)$ of an AC-subset of the conic $\C$ in $\PG(2,q)$
holds:
\begin{align}\label{eq1_newBnd}
t(q)<\Theta(q).
\end{align}
\end{theorem}

Similarly to \cite{St_1992_ComplNRC}, we use upper bounds on  $t(q)$ to prove the completeness of the normal rational curves as arcs in projective spaces.
From Theorem \ref{th1_main_bnd} and \cite[Theorems 3.1,3.5]{St_1992_ComplNRC} we obtained Corollaries \ref{cor1_main_NRC} and \ref{cor1_q^odd_NRC}; see Section \ref{ses_NRC}.
\begin{corollary}\label{cor1_main_NRC}
 Let
\begin{align}\label{eq1_new_copl}
 3\le N\le q+2-\Theta(q).
  \end{align}
  Then every normal rational curve in $\PG(N,q)$ is a complete $(q + 1)$-arc.
\end{corollary}
\begin{corollary}\label{cor1_q^odd_NRC}
Let $h\ge 1$ be a fixed integer. Let $p_0(1)=757$, $p_0(2)=1399$, $p_0(3)=2129$, $p_0(4)=2887$,
 $p_0(5)=3623$. Also, for $h\ge6$ let $p_0(h)$ be the smallest odd prime $\widehat{p}$ satisfying
\begin{align}\label{eq1_q^odd}
       \sqrt{\widehat{p}}>4c\sqrt{(2h+1)\ln \widehat{p}}\,+\frac{29}{4\widehat{p}^{\,\,h-0.5}}-\frac{20}{\widehat{p}^{\,\,h+0.5}}\,\,,
\end{align}
where  $c=1.62$ for $6\le h\le19$, $c=1.635$ for $20\le h\le28$,
$c=1.835$ for $h\ge29$.

Then  for every odd prime $p\ge p_0(h)$ in $\PG(N,q)$ with $q=p^{2h+1}$, $2\le N\le q-2$,
every normal rational curve   is a complete $(q+1)$-arc.
\end{corollary}

\begin{remark}
   In \eqref{eq1_q^odd}, the term $\frac{29}{4\widehat{p}^{\,\,h-0.5}}-\frac{20}{\widehat{p}^{\,\,h+0.5}}$ quickly decreases when $h$ grows.
   Therefore, practically, use of inequality\,\, $\widehat{p}>16c^2(2h+1)\ln \widehat{p}$\,\,
gives the same result as for~\eqref{eq1_q^odd}. In particular, we have checked this for $h\le16$.
\end{remark}

In Section \ref{sec_upper bound} we consider also implicit upper bounds on $t(q)$.

All bounds on $t(q)$ obtained in this paper are better than the bound of \eqref{eq1_Kovacs}.

Corollaries  \ref{cor1_main_NRC} and \ref{cor1_q^odd_NRC} extend regions of pairs $(N,q)$ for which it is proved that every normal rational curve in $\PG(N,q)$ is a complete $(q+1)$-arc.

Corollary  \ref{cor1_main_NRC} improves the  results of \cite[Theorem 9.2]{KorStSz_1997_SpFilSubNRC}, cf. \eqref{eq1_new_copl} and rows 10--11 of Table \ref{tab1}; in \eqref{eq1_new_copl} the region on $N$ values is greater  by $\thicksim 0.8\sqrt{q\ln q}$.

  Corollary \ref{cor1_q^odd_NRC} gives essentially smaller values $p_0(h)$ than \cite[Theorem 3.5]{St_1992_ComplNRC}. By  Corollary \ref{cor1_q^odd_NRC}, we have
$\{p_0(1),p_0(2),\ldots,p_0(16)\}=\{757,1399,2129,2887,3623,4621,5417,6247,$ $ 7079,
  7919,8779,9629,10499,11383,12253,13147\}$.
For comparison, \cite[Theorem 3.5]{St_1992_ComplNRC}, see row~9 of Table \ref{tab1}, provides
$\{p_0(1),p_0(2),\ldots,p_0(16)\}=\{16831,29663,43037,56747,70769,$ $85009, 99431,
   114031, 128767, 143651,158647,173741,188953,204251,219629, 235091\}$.
\section{The number of new covered points in one step of a step-by-step algorithm constructing AC-subsets}\label{sec_alg}

 Assume that  an AC-subset  is constructed by a step-by-step algorithm
(Algorithm, for short) adding a new point to the subset on
every step. As an example, we mention the greedy algorithm that
on every step adds to the subset a point providing the maximal
possible (for the given step) number of new covered points.

Let $w>0$ be a fixed integer. Consider the $(w+1)$st step of Algorithm.
This step  starts from a $w$-subset $\K_{w}\subset \C$ constructed in the previous
$w$ steps.
Let $\U(\K_{w})$ be the subset of points
of $\M_q$ not covered by the subset $\K_{w}$.

Let the subset $\K_{w}$
consist of $w$ points $A_{1},A_{2},\ldots ,A_{w}$.
Let $A_{w+1}\in \C\setminus \K_{w}$ be the point that
will be included into the subset in the $(w+1)$st step.
Denote by
$\U(\K_{w}\cup \{A_{w+1}\})$ the subset of points of $\M_q$ not
covered by the new subset $\K_{w}\cup \{A_{w+1}\}$.

Let $\overline{AB}$ be the line through points $A$ and $B$. The point $A_{w+1}$ defines a bundle $\B(A_{w+1})=\{\overline{A_{1}A_{w+1}},\overline{A_{2}A_{w+1}},\ldots,\overline{A_{w}A_{w+1}}\}$ of
$w$ tangents (unisecants)  to $\K_{w}$ which are bisecants of $\C$.
In order to obtain the next subset
$\K_{w+1}$, we may include to $\K_{w}$ any of $q+1-w$ points of $\C\setminus \K_{w}$.
So, there exist $q+1-w$ distinct points $A_{w+1}$ and $q+1-w$ distinct bundles.  Introduce the set of $w(q+1-w)$ lines
\begin{equation*}
 \B_{w+1}^{\,\cup} =\bigcup_{A_{w+1}\in \C\setminus \K_{w}}\B(A_{w+1}).
\end{equation*}
Let $P_{w+1}^{\,\cup}$ be the point multiset consisting of all points of $\B_{w+1}^{\,\cup}.$ A point that is the intersection of $m$ lines of $\B_{w+1}^{\,\cup}$ has multiplicity $m$ in $P_{w+1}^{\,\cup}$.

Let $\Delta(A_{w+1})$ be the number of the \emph{new covered} points in the $(w+1)$st step. Denote by $\N(A_{w+1})$ the set of \emph{new} points \emph{covered} by $\K_{w}\cup \{A_{w+1}\}$. By definition,
\begin{align*}
&\N(A_{w+1})=\U(\K_{w})\setminus\U(\K_{w}\cup \{A_{w+1}\}),\notag\\
&\Delta(A_{w+1})=\#\N(A_{w+1})=\#\U(\K_{w})-\#\U
(\K_{w}\cup \{A_{w+1}\}).  \label{eq2_Delta_w}
\end{align*}
 Introduce the point multiset
\begin{equation*}
  \N_{w+1}^{\,\cup} =\bigcup_{A_{w+1}\in \C\setminus \K_{w}}\N(A_{w+1})\subset P_{w+1}^{\,\cup}.
\end{equation*}
By the definitions above,
\begin{equation*}
     \#\N^{\,\cup}_{w+1}=\sum_{A_{w+1}\in \C\setminus \K_{w}}\Delta(A_{w+1}).
\end{equation*}

Let $P\in \U(\K_{w})\subset \M_q$ be a point  not covered by $\K_{w}$. Every point of $\M_q$ lies at most on two tangents of $\C$. The rest of lines through this point and the points of $\C$ are bisecants. Therefore, among the $w$ lines connecting $P$ with $\K_{w}$ there are at least $w-2$ bisecants of $\C$. None of those bisecants is a bisecant of $\K_{w}$ otherwise the point $P$ would be covered. Hence, all bisecants of $\C$ through $P$ and $\K_{w}$ belong to $\B_{w+1}^{\,\cup}$.
It means that
\emph{every point of $\U(\K_{w})$ is included in $\N^{\,\cup}_{w+1}$ at least $w-2$ times}. So,
\begin{equation}\label{eq2_card_multis}
    \#\N^{\,\cup}_{w+1}\ge (w-2)\cdot\#\U(\K_{w}).
\end{equation}

\begin{remark} For even $q$, every point of $\M_q$ lies  on one tangent of $\C$. Therefore for even $q$, in relation \eqref{eq2_card_multis}
we may change $w-2$ by $w-1$. Also, for odd $q$, an internal point does not belong to any tangent of a conic whereas each of the $\frac{1}{2}q(q+1)$
external points lies on two distinct tangents. Hence for odd $q$, in \eqref{eq2_card_multis} we may change $(w-2)\cdot\#\U(\K_{w})$ by $(w-2)\cdot\#\U(\K_{w})+2\max\{0,\#\U(\K_{w})-\frac{1}{2}q(q+1)\}$. These changes could slightly improve estimates below. However, for simplicity of presentation, we save relation \eqref{eq2_card_multis} as it is.
\end{remark}

By the above, the average number, say $\Delta _{w+1}^{\text{aver}}$, of new covered points in a bundle in the $(w+1)$st step is as follows
\begin{equation*}\label{eq_aver_def}
    \Delta _{w+1}^{\text{aver}}=\frac{\sum\limits_{A_{w+1}\in \C\setminus \K_{w}}\Delta (A_{w+1})}{q+1-w}\ge\frac{(w-2)\cdot\#\U(\K_{w})}{q+1-w}.
\end{equation*}

Clearly,
\begin{equation*}\label{eq2_max>=aver}
   \max\limits_{A_{w+1}\in\, \C\setminus \K_{w}}\Delta(A_{w+1})\ge\left\lceil\Delta _{w+1}^{\text{aver}}\right\rceil.
\end{equation*}
So, we have proved the following lemma.
\begin{lemma}
    For an arbitrary step-by-step algorithm, there exists a point $A_{w+1}$ providing
\begin{equation}\label{eq2_>=aver}
     \Delta(A_{w+1})\ge\left\lceil\frac{(w-2)\cdot\#\U(\K_{w})}{q+1-w}\right\rceil.
\end{equation}
\end{lemma}
Note that the greedy algorithm always finds the  point $A_{w+1}$ with property \eqref{eq2_>=aver}.

\section{Upper bounds on the smallest size of an AC-subset based on properties of step-by-step algorithms}\label{sec_upper bound}
We denote
\begin{align*}
   t^*(q)=\frac{t(q)}{\sqrt{q\ln q}}\,\,.
\end{align*}
Let $t(q)<f(q)$. Then $t^*(q)<f(q)/\sqrt{q\ln q}$. The upper bounds on $t^*(q)$ are more convenient for graphical representation than bounds on $t(q)$.
If $f(q)$ is called ``Bound L'', say, then we call $f(q)/\sqrt{q\ln q}$ ``Bound L*''.
\subsection{Implicit bound A}
By Section \ref{sec_alg},
\begin{align}\label{eq3_|U_w+1|}
\#\U(\K_{w}\cup \{A_{w+1}\})=\#\U(\K_{w})-\Delta(A_{w+1})
\le \#\U(\K_{w})-\left\lceil\frac{(w-2)\cdot\#\U(\K_{w})}{q+1-w}\right\rceil.
\end{align}
Define $U_w$ as an upper bound on $\#\U(\K_{w})$:
\begin{align}\label{eq3_|U_w|<U_w}
\#\U(\K_{w})=U_w-\delta\le U_w;\quad \delta\ge0.
\end{align}
 By \eqref{eq3_|U_w+1|}, \eqref{eq3_|U_w|<U_w},
\begin{align*}
&\#\U(\K_{w}\cup \{A_{w+1}\})
\le U_{w}-\delta-\left\lceil\frac{(w-2)(U_{w}-\delta)}{q+1-w}\right\rceil=\\
&U_{w}-\left\lceil\frac{(w-2)U_{w}+(q+3-2w)\delta}{q+1-w}\right\rceil.
\end{align*}

From now on, we suppose
\begin{equation}\label{eq3_condition}
    q+3>2w.
\end{equation}


Under condition \eqref{eq3_condition}, it holds that
\begin{equation}\label{eq3_|U_w+1|final}
    \#\U(\K_{w}\cup \{A_{w+1}\})=\#\U(\K_{w})-\Delta(A_{w+1})\le U_w-\left\lceil\frac{(w-2)U_{w}}{q+1-w}\right\rceil.
\end{equation}

Assume that there exists a $w_0$-subset $\K_{w_{0}}\subset\C\subset PG(2,q)$ that does not cover at most
$U_{w_{0}}$ points of $\M_q$. Then, starting from values $w_0$ and $U_{w_{0}}$, one can iteratively apply
the relation \eqref{eq3_|U_w+1|final} and obtain eventually $\#\U(\K_{w}\cup\{A_{w+1}\})=0$ for some $w$, say $w_\text{fin}$.
Clearly, $w_\text{fin}$ depends on $w_0$ and $U_{w_{0}}$, i.e. we have a function $w_\text{fin}(w_0,U_{w_0})$.
The size $k$ of the obtained AC-subset is as follows:
\begin{equation*}
k= w_\text{fin}(w_0,U_{w_0})+1  \text{ under condition }\#\U(\K_{w_\text{fin}(w_0,U_{w_0})}\cup\{A_{w_\text{fin}(w_0,U_{w_0})+1}\})=0.
\label{eq3_k=w+1}
\end{equation*}
From the above we have the following theorem.
\begin{theorem}
\label{th3_bndA} \textbf{\emph{(}implicit bound $A(w_0,U_{w_0})$\emph{)}} Let the values $w_0$, $U_{w_{0}}$, and $w_\text{fin}(w_0,U_{w_0})$ be defined
and calculated as above. Let also $w_\text{fin}(w_0,U_{w_0})<\frac{q+3}{2}$. Then it holds
that
\begin{equation*}
t(q)\leq w_\text{fin}(w_0,U_{w_0})+1. \label{eq3_boundA}
\end{equation*}
\end{theorem}

It is easily seen that, for any $q$, there exists a $5$-subset $\K_5\subset\C\subset PG(2,q)$ that does not cover
$\#\U(\K_{5})=\#\M_q-(10q-25)\le U_{5}=(q-5)^2$ points of $\M_q$. The corresponding implicit bound $A^*(5,(q-5)^2)$ (i.e. the value
$(w_\text{fin}(5,(q-5)^2)+1)/\sqrt{q\ln q}$\,) is shown by the third blue curve on Figs. \ref{fig_1} and \ref{fig_2}.
\begin{figure}[htbp]
\includegraphics[width=\textwidth]{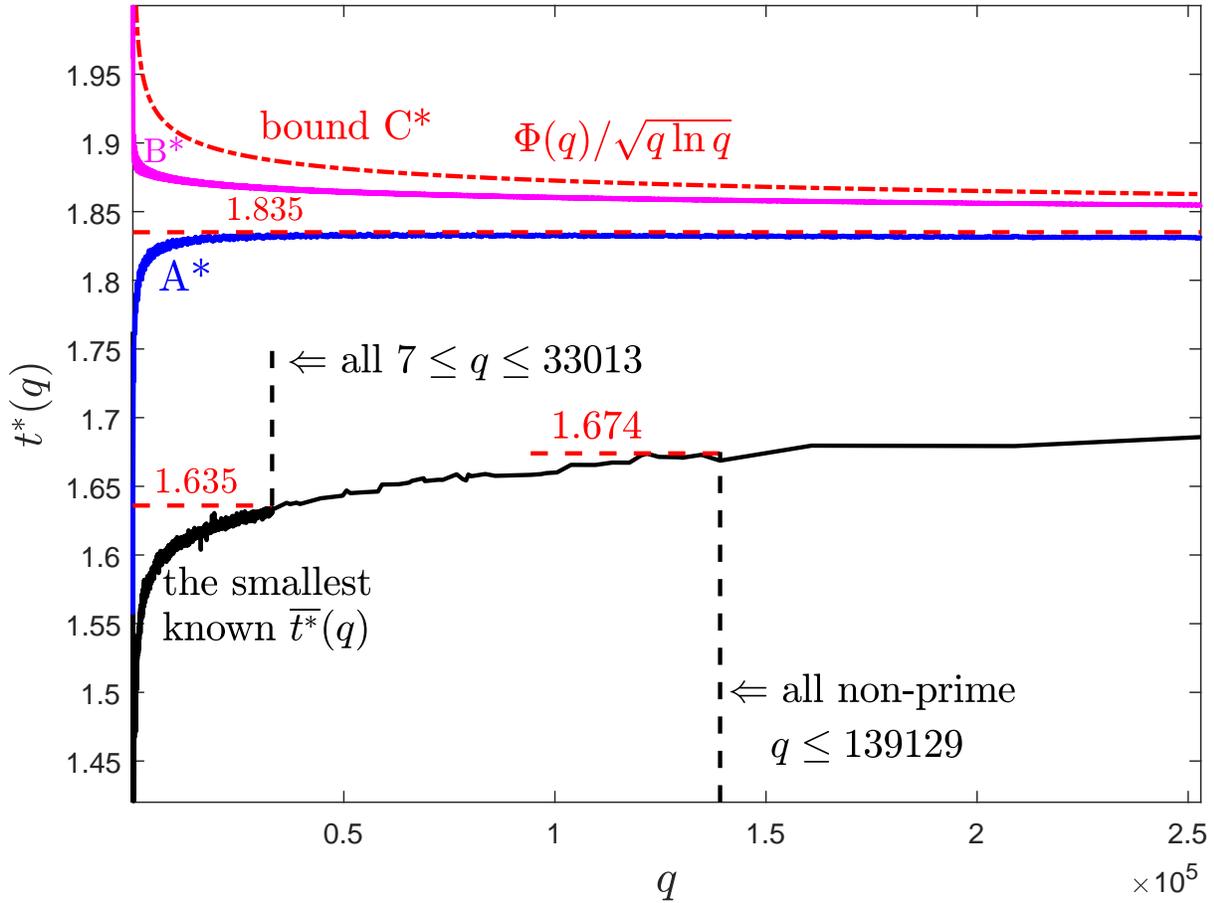}
\caption{\textbf{Upper bounds on sizes
 of AC-subsets divided by $\sqrt{q\ln q}$, $q\le253009$:} bound~C* equal to $\Phi(q)/\sqrt{q\ln q}$ (\emph{top dashed-dotted red curve});
 implicit bound B* (\emph{the 2-nd magenta curve}); bound \eqref{eq3_1835} (\emph{dashed red line $y=1.835$});
 implicit bound $A^*(5,(q-5)^2)$ (\emph{the 3-rd blue curve});
  bound \eqref{eq4_1638} (\emph{dashed red line $y=1.635$}); bound \eqref{eq4_1673} (\emph{dashed red line $y=1.674$}); the smallest known sizes of AC-subsets divided by $\sqrt{q\ln q}$, i.e. values $\overline{t^*}(q)$ (\emph{bottom black curve}). Vertical dashed lines $x=33013$ and $x=139129$ mark regions of complete computer search, respectively, for all  prime powers $q$ and all non-prime $q$'s}\label{fig_1}
\end{figure}
\begin{figure}[htbp]
\includegraphics[width=\textwidth]{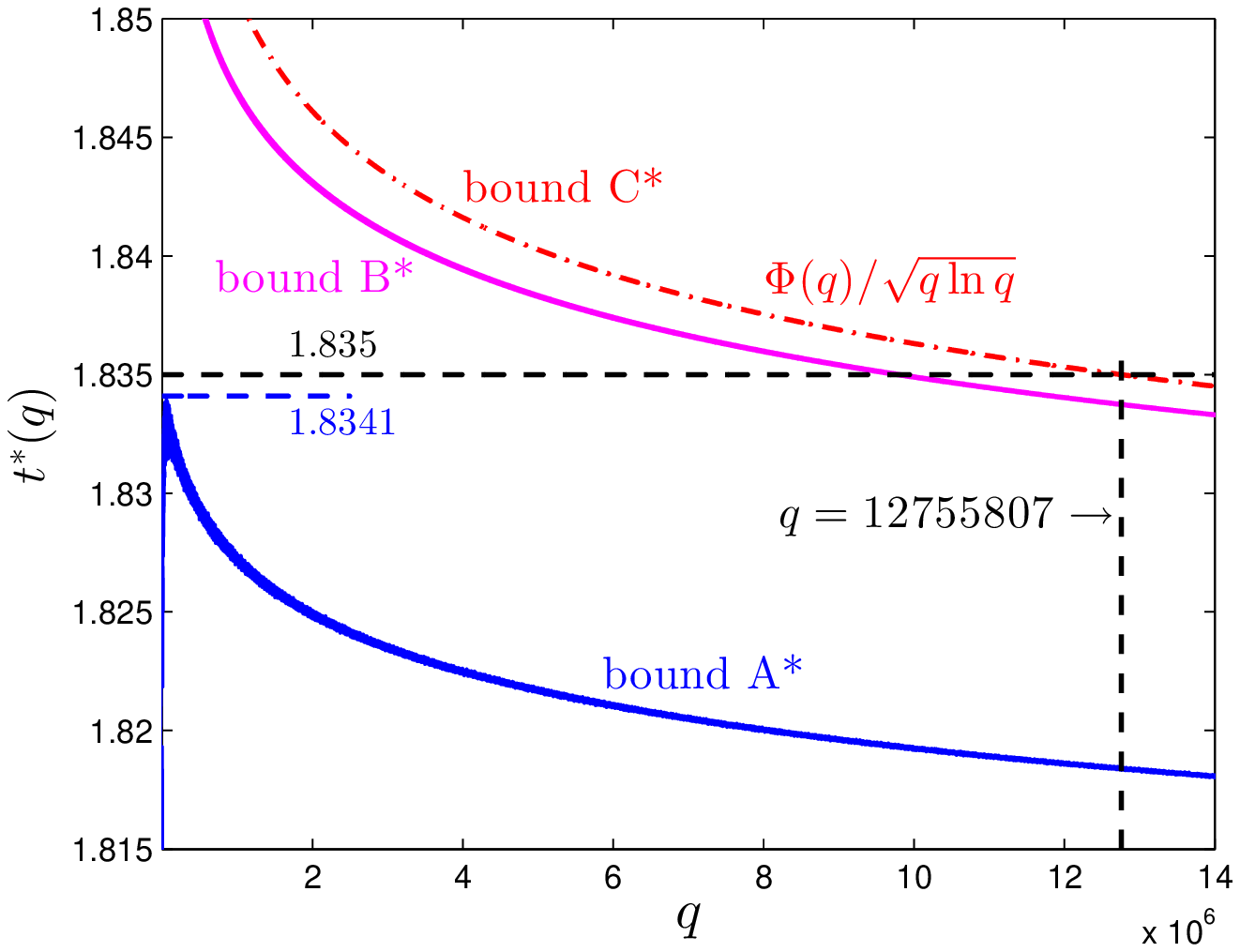}
\caption{\textbf{Upper bounds on sizes
 of AC-subsets divided by $\sqrt{q\ln q}$, $q\le14000029$:} bound~C* equal to $\Phi(q)/\sqrt{q\ln q}$ (\emph{top dashed-dotted red curve});
 implicit bound B* (\emph{the 2-nd magenta curve}); bound \eqref{eq3_1835} (\emph{dashed red line $y=1.835$});
 implicit bound $A^*(5,(q-5)^2)$ (\emph{the 3-rd blue curve})}\label{fig_2}
\end{figure}
\begin{observation}\label{observ3}
    In the region $7\le q\le 55711$ the implicit bound $A^*(5,(q-5)^2)$ tends to increase with the maximal value $A^*(5,(q-5)^2)\thicksim1.8341$ for $q=55711$.
    In the region $55711< q\le 14000029$ the  bound $A^*(5,(q-5)^2)$ tends to decrease with the minimal value $A^*(5,(q-5)^2)\thicksim1.8180$ for $q=13995829$, see Fig. \emph{\ref{fig_2}}.
\end{observation}

\subsection{A truncated iterative process}
From \eqref{eq3_|U_w+1|} we have that
\begin{align}\label{eq3_U_w+1-lrceil}
&\#\U(\K_{w}\cup \{A_{w+1}\})=\#\U(\K_{w})-\Delta(A_{w+1})
\le U_w\left( 1-\frac{w-2}{q+1-w}\right).
\end{align}
Clearly, $\#\U(\K_{1})\le U_1= q^2.$ Using \eqref{eq3_U_w+1-lrceil}
iteratively, we obtain
\begin{align}
\#\U(\K_{w}\cup\{A_{w+1}\})=U_{w+1}\leq q^2f_q(w)  \label{eq3_N_w+1<_},
\end{align}
where
\begin{align}\label{eq3_fq_main}
f_q(w)=\prod_{i=1}^{w}\left( 1-\frac{i-2}{q+1-i}\right).
\end{align}

From now on, we will stop the iterative process when $\#\U(\K_{w}\cup\{A_{w+1}\})\leq \xi  $ where $\xi
\geq 1$ is some value that we may assign to improve estimates. Note that if some point $P\in \M_q$ is not covered by $\K_{w}\cup\{A_{w+1}\}$, one always can find a point $A_{w+2}\in\C\setminus(\K_{w}\cup\{A_{w+1}\})$ such that $P$ is covered by $\K_{w}\cup\{A_{w+1},A_{w+2}\}$. It means that after the end of the iterative process we can add at most $\xi$ points of $\C$ to the running subset in order to get a  $k$-AC-subset with size $k$  satisfying
\begin{equation}
w+1\leq k\leq w+1+\xi  \text{ under condition }\#\U(\K_{w}\cup\{A_{w+1}\})\leq \xi  .
\label{eq3_k=w+1+xi}
\end{equation}

\begin{theorem}
\label{th3_main}   Let $\xi\ge1$ be a fixed value independent of $w$. Let $w<\frac{q+3}{2}$  satisfy
\begin{equation}
f_{q}(w)=\prod_{i=1}^{w}\left( 1-\frac{i-2}{q+1-i}\right)\leq \frac{\xi }{q^2}.  \label{eq3_main}
\end{equation}
Then it holds that
\begin{equation}
t(q)\leq w+1+\xi.  \label{eq3_general bound}
\end{equation}
\end{theorem}
\begin{proof}
By \eqref{eq3_N_w+1<_}, to provide the inequality $\#\U(\K_{w}\cup\{A_{w+1}\})\leq
\xi $ it is sufficient to find $w$ such that $q^2f_{q}(w)\leq \xi $.
Now \eqref{eq3_general
bound} follows from \eqref{eq3_k=w+1+xi}.
\end{proof}

Clearly, we should choose $\xi $ such that $w+1+\xi $ is small under
condition $\#\U(\K_{w}\cup\{A_{w+1}\})\leq \xi$.

In order to get more simple forms of upper bounds on $t(q)$ we will find an upper bound on $f_q(w)$ of \eqref{eq3_fq_main}.
To this end we  use the Taylor series  $e^{-\alpha }=1-\alpha +\frac{\alpha
^{2}}{2}-\frac{\alpha ^{3}}{6}+...,$ whence
\begin{equation}\label{eq3_Taylor}
1-\alpha <e^{-\alpha }\mbox{ for }\alpha \neq 0.
\end{equation}

\subsection{Implicit bound B}
\begin{lemma}
    It holds that
    \begin{equation}\label{eq3_estim_Stef}
       f_{q}(w)=\prod_{i=1}^{w}\left( 1-\frac{i-2}{q+1-i}\right)<e^{-S},
    \end{equation}
    where
    \begin{align}\label{eq3_estim_Stef_det}
-w+(q-1)\ln\frac{q+1}{q+1-w}<S<-w+(q-1)\ln\frac{q}{q-w}.
    \end{align}
\end{lemma}
\begin{proof}
By \eqref{eq3_Taylor},
\begin{equation*}
    \prod_{i=1}^{w}\left( 1-\frac{i-2}{q+1-i}\right)<e^{-S}, \quad S=\sum\limits_{i=1}^{w}\frac{i-2}{q+1-i}.
\end{equation*}
Also,
\begin{align*}
 & S=\sum_{i=1}^{w}\frac{i-2}{q+1-i}= \sum_{u=-1}^{w-2}\frac{u}{q-1-u}=-w+\sum_{u=-1}^{w-2}\left(\frac{u}{q-1-u}+1\right)=\notag\\
& -w+(q-1)\sum_{u=-1}^{w-2}\frac{1}{q-1-u}= -w+(q-1)\sum_{t=q+1-w}^{q}\frac{1}{t}.
\end{align*}
  It is well known that
  \begin{equation*}
    \ln(q+1)<\sum_{t=1}^{q}\frac{1}{t}<1+\ln q.
  \end{equation*}
  Therefore,
  \begin{align*}
  \ln(q+1)-\ln(q+1-w)<\sum_{t=q+1-w}^{q}\frac{1}{t}=\sum_{t=1}^{q}\frac{1}{t}-\sum_{t=1}^{q-w}\frac{1}{t}<\ln q-\ln(q-w).~~~~ \qedhere
  \end{align*}
\end{proof}
\begin{corollary}\label{cor3_impl_B_gen}
                     Let $\xi\ge1$ be a fixed value independent of $w$. Let $w<\frac{q+3}{2}$ satisfy
\begin{equation*}
w-(q-1)\ln\frac{q+1}{q+1-w}\leq \ln\frac{\xi }{q^2}.
\end{equation*}
Then it holds that
\begin{equation*}
t(q)\leq w+1+\xi.
\end{equation*}
\end{corollary}
\begin{proof} We substitute \eqref{eq3_estim_Stef} and \eqref{eq3_estim_Stef_det} in \eqref{eq3_main}.
\end{proof}
\begin{corollary}\label{cor3_implB} \textbf{(implicit bound B)}
                    Let $w<\frac{q+3}{2}$  satisfy
\begin{equation*}
w-(q-1)\ln\frac{q+1}{q+1-w}\leq \ln\frac{1}{q\sqrt{3q\ln q}}.
\end{equation*}
Then it holds that
\begin{equation*}
t(q)\leq w+1+\sqrt{\frac{q}{3\ln q}}.
\end{equation*}
\end{corollary}
\begin{proof} In the assertions of Corollary \ref{cor3_impl_B_gen}, we use $\xi =\sqrt{\frac{q}{3\ln q}}$.
\end{proof}

The implicit bound~B*  is shown by the second magenta curve on Figs. \ref{fig_1} and \ref{fig_2}.

\subsection{Explicit bounds}

By \eqref{eq3_fq_main} and \eqref{eq3_Taylor}, we have
\begin{equation}
f_q(w)<\prod_{i=1}^{w}\left( 1-\frac{i-2}{q}\right) <\prod_{i=1}^{w}e^{-(i-2)/q}=e^{-(w^{2}-3w)/2q}<e^{-(w-2)^{2}/2q}.
\label{eq3_fqw_approx}
\end{equation}
\begin{lemma}
\label{lem3_basic}  Let $\xi\ge1$ be a fixed value independent of $w$. The value
\begin{equation}
\frac{q+3}{2}>w\geq \sqrt{2q}\sqrt{\ln \frac{q^2}{\xi }}+3  \label{eq3_w<}
\end{equation}
satisfies  inequality \eqref{eq3_main}.
\end{lemma}

\begin{proof}
By (\ref{eq3_fqw_approx}), to provide (\ref{eq3_main}) it
is sufficient to find $w$ such that
\begin{equation*}
e^{-(w-2)^{2}/2q}< \frac{\xi }{q^2}.
\end{equation*}
As $w$ should be an integer, in \eqref{eq3_w<} one is added. Inequality $w<\frac{q+3}{2}$ is obvious.
\end{proof}

\begin{theorem}
\label{th3_general bound_converse}In $
PG(2,q)$ it holds that
\begin{equation}
t(q)\leq \sqrt{2q}\sqrt{\ln \frac{
q^2}{\xi }}+\xi+4 ,~~\xi \geq 1,  \label{eq3_gen bound}
\end{equation}
where $\xi $ is an arbitrarily chosen value.
\end{theorem}

\begin{proof}
The assertion follows from \eqref{eq3_general bound} and \eqref{eq3_w<}.
\end{proof}

\begin{remark}\label{rem3_xi_deriv}
     We consider
the function of $\xi$ of the form
\begin{equation*}
\phi(\xi)=\sqrt{2q}\sqrt{\ln \frac{q^{2}}{\xi }}+\xi+4.  \label{eq2_phi}
\end{equation*}
Its derivative by $\xi$ is
\begin{equation*}
\phi'(\xi)=1-\frac{1}{\xi}\sqrt{\frac{q}{2\ln\frac{q^{2}}{\xi }}}.  \label{eq2_der phi}
\end{equation*}
Put $\phi'(\xi)=0$. Then
\begin{align} \label{eq4_xi^2}
    \xi^2=\frac{q}{4\ln q-2\ln \xi}.
\end{align}
We find $\xi$ in the form $
\xi=\sqrt{\frac{q}{c\ln q}}$. By \eqref{eq4_xi^2}, $c=3+\frac{\ln c+\ln\ln q}{\ln q}.$
For simplicity, we choose $c=3$. Then $
\xi=\sqrt{\frac{q}{3\ln q}}$ and the value
\begin{align*}
\phi'\left(\sqrt{\frac{q}{3\ln q}}\,\right)=1-\sqrt{\frac{3\ln
q}{3\ln q+\ln \ln q+\ln
3}}
\end{align*}
 is close to zero for growing $q$. Also, it is easy to check the following: $\phi'(1)<0$ if $q\ge9$,
$\phi'(\xi)$ is an increasing function,
$
0<\phi'\left(\sqrt{\frac{q}{3\ln q}}\,\right)<\phi'(\sqrt{q})=1-\sqrt{\frac{1}{3\ln q}}.
$

 So, the choice
$\xi =\sqrt{\frac{q}{3\ln q}} $ in \eqref{eq3_gen bound} seems
to be convenient.
\end{remark}

\begin{theorem}\label{th3_explicit bndC} \textbf{(Bound C)} The following upper
bound on the smallest size $t(q)$ of an AC-subset in $\PG(2,q)$
holds.
\begin{align}\label{eq3_bnd_C}
t(q)<\Phi(q)=\sqrt{q(3\ln q+\ln\ln q +\ln3)}+\sqrt{\frac{q}{3\ln q}}+4\thicksim\sqrt{3q\ln q}.
\end{align}
\end{theorem}
\begin{proof}
We substitute $
\xi =\sqrt{\frac{q}{3\ln q}}
$ in \eqref{eq3_gen bound}.
\end{proof}

The bound C* (i.e. the value $\Phi(q)/\sqrt{q\ln q}$\,) is shown by the top dashed-dotted red curve on Figs. \ref{fig_1} and \ref{fig_2}.
\begin{remark}\label{rem3_xi=1_xi=sqrt_q}
    If in \eqref{eq3_gen bound} we take $\xi=1$ and $\xi =\sqrt{q}$, we obtain bounds \eqref{eq3_xi=1}  and~\eqref{eq3_xi=sqrt_q}.
        \begin{equation}\label{eq3_xi=1}
            t(q)<2\sqrt{q\ln q}+5.
        \end{equation}
        \begin{equation}\label{eq3_xi=sqrt_q}
            t(q)<\sqrt{3q\ln q}+\sqrt{q}+4.
        \end{equation}
        It can be shown that bounds \eqref{eq3_xi=1}  and \eqref{eq3_xi=sqrt_q} are worse than \eqref{eq3_bnd_C}.

If we put, see Remark \ref{rem3_xi_deriv},
$ c=3+\frac{1+\ln\ln q}{\ln q}$, $\xi =\sqrt{\frac{q}{3\ln q+\ln\ln q+1}} $, we improve bound  \eqref{eq3_bnd_C}. But,  the improvement is unessential whereas the  bound takes a lengthy form.
\end{remark}

\begin{theorem}\label{th3_1835}  The following upper
bound on the smallest size $t(q)$ of an AC-subset in $\PG(2,q)$
holds.
\begin{align}\label{eq3_1835}
t(q)<1.835\sqrt{q\ln q}.
\end{align}
\end{theorem}
\begin{proof}  For $q\le 12755807$ we checked by computer that the
implicit bound $A(5,(q-5)^2)<1.8341\sqrt{q\ln q}$; so in this region the assertion  is provided by the bound $A(5,(q-5)^2)$,
see Observation \ref{observ3} and Fig.\,\ref{fig_2}.
It is easy to see that $\Phi(q)/\sqrt{q\ln q}$ is a decreasing function of~$q$. Moreover, $\Phi(q)/\sqrt{q\ln q}<1.835$ for $q=12755807$.
So, for $q>12755807$ the assertion is provided by the bound~C.
\end{proof}
The  bound \eqref{eq3_1835} is presented by the dashed red line $y=1.835$ in Figs. \ref{fig_1} and \ref{fig_2}.

\section{Computer assisted results on $t(q)$ and $t^*(q)$}\label{sec_comput}

Let $\overline{t}(q)$ be the smallest \emph{known} size of an AC-subset in $\PG(2,q)$. Let
$\overline{t^*}(q)=\overline{t}(q)/\sqrt{q\ln q}.$
We denote the following sets of values of $q$:
$Q_2:=\{5\le q\le 33013,~ q \text{ prime power}\};$
$Q_3:=\{5\le q\le 32,~ q \text{ prime power}\};$
$Q_4:=Q_1\cup\{160801,208849,253009\}$. Let $Q_1$ be as in \eqref{eq1_region_Q1}.

For the set $Q_3$ we obtained by computer search the smallest sizes
$t(q)$ of AC-subsets of $\C$ in $\PG(2,q)$, see Table \ref{tabQ3}.
The algorithm, used in the search,
fixes a conic, computes all the non-equivalent point subsets of the conic of a certain size (6 in our complete cases)
and extends each of them trying to obtain a minimal AC-subset. Each time an example is found only smaller examples are looked for.
Minimality is checked explicitly:
once we have found an AC-subset we test that deleting from it a point in all possible ways no almost complete subset is obtained.
All computations are performed using the system for symbol calculations MAGMA \cite{Magma}.
\begin{table}[htbp]
\caption{ The smallest sizes $t(q)$ of AC-subsets of $\C$ in $\PG(2,q)$, $q\in Q_3$ }
 \renewcommand{\arraystretch}{1.00}
 \begin{center}
 \begin{tabular}
{@{}rr|rr|rr|rr|rr|rr|rr|rr@{}}
 \hline
 $q$ &$t(q)$& $q$ &$t(q)$& $q$ &$t(q)$& $q$ &$t(q)$& $q$ &$t(q)$& $q$ &$t(q)$& $q$ &$t(q)$& $q$ &$t(q)$\\
 \hline
5&5&7&6&8&6&9&6&11&8&13&8&16&9&17&10\\
19&11&23&12&25&12&27&13&29&13&31&14&32&15\\
\hline
 \end{tabular}\label{tabQ3}
 \end{center}
 \end{table}

 For the sets $Q_2$ and $Q_4$ we obtained small AC-subsets of $\C$ in $\PG(2,q)$ by computer search\footnote{The computer
search for $q\in
Q_2\cup Q_4$ has been carried out using computing resources of the
federal collective usage center Complex for Simulation and Data
Processing for Mega-science Facilities at NRC ``Kurchatov Institute",
http://ckp.nrcki.ru/.}. For it we  used step-by-step randomized greedy algorithms
similar to those from \cite{BDFKMP-JG2016}, see also the references therein. Recall that at each step a randomized greedy algorithm
maximizes some objective
function $f$, but some steps are executed in a random manner. Also, if one and the same maximum
of $f$ can be obtained in different ways, the choice is made at random. As the value of the objective
function, the number of points lying on bisecants of the running subset is considered.

As far as the authors know, sizes of AC-subsets, obtained by the mentioned computer search, are the smallest known. The corresponding values of $\overline{t^*}(q)$ are shown by the bottom black curve in Fig.\,\ref{fig_1}. Recall that,
\begin{align*}
   t^*(q)=\frac{t(q)}{\sqrt{q\ln q}}\,\,.
\end{align*}
The values $\overline{t}(q)$ and  $\overline{t^*}(q)$ for $q\in Q_4$ and prime $q\in Q_2$  are given in Tables 3 and 4, respectively, see Appendix.
As values of $\overline{t^*}(q)$ are not integers, in Tables 3 and 4 we give rounded values of $\overline{t^*}(q)$, moreover we round up.  This explains
the entry ``\,$\overline{t^*}(q)<$'' in the top of columns.

In Table 4, the values $\overline{t^*}(q)$ are written for not all $q$'s. The rules for entries $\overline{t^*}(q)$ are as follows. Assume that the following holds: $q'<q''$; the values $\overline{t^*}(q')$ and $\overline{t^*}(q'')$ are written in Table 4; no value $\overline{t^*}(q)$ is written in the table if $q'<q<q''$.  Then  $\overline{t^*}(q')\le\overline{t^*}(q'')$ and $\overline{t^*}(q)\le\overline{t^*}(q')$ with $q'<q<q''$.

For example, one may take $q'=19$ and $q''=307$. We see that no value $\overline{t^*}(q)$ is written in Table 4 if $19<q<307$. We have $\overline{t^*}(19)\approx1.471<\overline{t^*}(307)\approx1.479$ and $\overline{t^*}(q)\le1.471$ with $19<q<307$.

So, in Table 4, the blank on place $\overline{t^*}(q)$ means that $\overline{t^*}(q)\le\overline{t^*}(q')$ under the conditions that $q'<q$, value $\overline{t^*}(q')$ is written in the table, and no  value $\overline{t^*}(q^\bullet)$ is written if $q'<q^\bullet<q$.

By computer search for the sets $Q_2$ and $Q_4$, see Tables 3 and 4, we have Theorem~\ref{th4_computBnd}.
\begin{theorem}\label{th4_computBnd}
The following upper
bounds on the smallest size $t(q)$ of an AC-subset of the conic $\C$ in $\PG(2,q)$
hold:
\begin{align}
& t(q)<1.525\sqrt{q\ln q}, &&8\le q\le 887,~ q \text{ prime power},~q\ne11;\displaybreak[2]\label{eq4_1525}\\
& t(q)<1.548\sqrt{q\ln q}, &&887< q\le 1553,~ q \text{ prime power};\displaybreak[2]\label{eq4_1548}\\
& t(q)<1.572\sqrt{q\ln q}, &&1553< q\le 2351,~ q \text{ prime power}, ~q=11;\displaybreak[2]\label{eq4_1572}\\
& t(q)<1.585\sqrt{q\ln q}, &&2351< q\le 4027,~ q \text{ prime power};\displaybreak[2]\label{eq4_1585}\\
& t(q)<1.620\sqrt{q\ln q}, && 4027< q\le 17041,~ q \text{ prime power};\displaybreak[2]\label{eq4_162}\\
& t(q)<1.635\sqrt{q\ln q}, && 17041< q\le 33013,~ q \text{ prime power},~q=7;\displaybreak[2]\label{eq4_1638}\\
&t(q)<1.674\sqrt{q\ln q}, && q=p^m,~p \text{ prime},~m\ge2,~q\in Q_1;\displaybreak[2]\label{eq4_1673}\\
&t(q)<1.686\sqrt{q\ln q}, && q=160801,208849,253009.\notag
\end{align}
\end{theorem}
The  bounds \eqref{eq4_1638}, \eqref{eq4_1673} are presented by dashed red lines $y=1.635$, $y=1.674$ in Fig.\,\ref{fig_1}.

\section{New  bounds on $t(q)$  and completeness of normal rational curves} \label{ses_NRC}

\emph{Proof of Corollary \ref{cor1_main_NRC}.}
  We substitute the new bounds of Theorem \ref{th1_main_bnd} in relation \eqref{eq1_Storm3.1} taken from \cite[Theorem 3.1]{St_1992_ComplNRC}.
 $\qed$

\emph{Proof of Corollary \ref{cor1_q^odd_NRC}.}
    We act analogously to the proof of \cite[Theorem 3.5]{St_1992_ComplNRC}, changing in it $6\sqrt{q\ln q}$ by $c\sqrt{q\ln q}$. As the result we obtain
    inequality \eqref{eq1_q^odd}.

    By \eqref{eq1_q^odd}, for $c=1.835$, $h\ge29$, we have $p_0(h)\ge 33079>33013$; but for $c=1.835$, $h\le28$, it holds that
    $p_0(h)\le31840<33013.$ So, by \eqref{eq3_1835} and \eqref{eq4_1525}--\eqref{eq4_1638}, we may take $c=1.835$ for $h\ge29$ and $c=1.635$ for $h\le28$.

    Again we use \eqref{eq1_q^odd}. For $c=1.635$, $h\ge20$, we have $p_0(h)\ge17091>17041$; but for $c=1.635$, $h\le19$, it holds that
    $p_0(h)\le16164<17041.$ So, by \eqref{eq4_1525}--\eqref{eq4_1638}, we may take $c=1.635$ for $20\le h\le28$ and $c=1.62$ for $h\le19$.

    Now for $h=1,\ldots,5$ we found $p_0(h)$ as a solution of \eqref{eq1_q^odd} taking $c$ on the base Theorem~\ref{th4_computBnd}.
     For the given $h$, at the beginning we obtain $p_0(h)$ with $c=1.62$. Then we decrease $c$ using \eqref{eq4_1525}--\eqref{eq4_1585} and get
     a smaller $p_0(h)$. For $c=1.62$ we obtain $p_0(1)=877$, $p_0(2)=1543$, $p_0(3)=2273$, $p_0(4)=3037$, $p_0(5)=3821$. So,
       we may put $c=1.525$ for $h=1$,  $c=1.548$ for $h=2$,
     $c=1.572$ for $h=3$, and
     $c=1.585$ for $h=4,5$, see  \eqref{eq4_1525}, \eqref{eq4_1548}, \eqref{eq4_1572}, and \eqref{eq4_1585}, respectively. Solutions of inequality \eqref{eq1_q^odd} for these $(c,h)$ are the values
     $p_0(1),\ldots,p_0(5)$ written in the assertion of the corollary.
 $\qed$

\begin{remark}
 We can also improve the result of \cite[Theorem 3.4]{St_1992_ComplNRC}.
   If in the proof of \cite[Theorem 3.4]{St_1992_ComplNRC} one uses the new bound $t(q)<1.835\sqrt{q\ln q}$
instead of~\eqref{eq1_Kovacs} then the following assertion can be proved:
\emph{for prime $p\ge 76207$ every normal rational curve in $\PG(N,p)$
with $2\le N\le p-2$ is a complete $(q+1)$-arc.}

For comparison note that in \cite[Theorem 3.4]{St_1992_ComplNRC} the value $p> 1007215$ is pointed out.

Of course, due to the results of \cite{Ball-book2015,Ball-JEMS2012,BallBeule-DCC2012}, see row 5 of Table \ref{tab1}, we know that for \emph{all primes $p$} normal rational curves in $\PG(N,p)$ are complete.
\end{remark}

The authors are grateful to participants of the Coding Theory seminar at the
Kharkevich Institute for Information Transmission Problems of the Russian Academy of Sciences
for the constructive and useful discussion of the work.

\newpage
\section*{Appendix. Tables of the smallest known sizes $\overline{t}(q)$ of AC-subsets in $\PG(2,q)$}
\textbf{Table 3. }The smallest known sizes $\overline{t}(q)$ of AC-subsets in $\PG(2,q)$ and values
$\overline{t^*}(q)$,\\ $q$ non-prime, $q\in \{8\le q\le 139129,~ q=p^m,~p \text{ prime},~m\ge2\}\cup\{160801,208849,253009\}$ \medskip

\noindent
\begin{tabular}
{@{}rrrc|rrrc|rrrc@{}}
 \hline
&&&$\overline{t^*}(q)$ &
&&&$\overline{t^*}(q)$ &
&&&$\overline{t^*}(q)^{ \vphantom{H^{L}}}$\\
$q$ & $p^m$ &$\overline{t}(q)$ &$<$&$q$ & $p^m$ & $\overline{t}(q)$ &$<$&$q$ & $p^m$ & $\overline{t}(q)$ &$<$\\
 \hline
 8 & $2^3$&6  &1.48&  9&$3^2$&  6  &1.35&  16&$2^4$ & 9 &1.36 \\
 25 & $5^2$&12  &1.34&  27 &$3^3$& 13  &1.38&  32 &$2^5$& 15&1.43  \\
 49 &$7^{2}$& 18  &1.31&  64 &$2^{6}$ &22  &1.35&  81&$3^4$ & 25&1.33  \\
 121 &$11^2$& 33  &1.37&  125&$5^3$ & 35  &1.43&  128 &$2^7$& 35&1.41  \\
 169 &$13^2$& 41  &1.40&  243 &$3^5$& 53  &1.46&  256&$2^8$ & 55& 1.46 \\
 289 &$17^2$& 58  &1.44&  343&$7^3$ & 66  &1.48&  361&$19^2$ & 66& 1.44 \\
\hline
 \end{tabular}
\newpage
\noindent \textbf{Table 3.} Continue\medskip

\noindent\begin{tabular}
{@{}rrrc|rrrc|rrrc@{}}
 \hline
&&&$\overline{t^*}(q)$ &
&&&$\overline{t^*}(q)$ &
&&&$\overline{t^*}(q)^{ \vphantom{H^{L}}}$\\
$q$ & $p^m$ &$\overline{t}(q)$ &$<$&$q$ & $p^m$ & $\overline{t}(q)$ &$<$&$q$ & $p^m$ & $\overline{t}(q)$ &$<$\\
 \hline
68921&$41^3$&1451&1.656&69169&$263^2$&1452&1.654&72361&$269^2$&1489&1.655\\
73441&$271^2$&1501&1.655&76729&$277^2$&1541&1.659&78125&$5^7$&1553&1.656\\
78961&$281^2$&1561&1.655&79507&$43^3$&1571&1.659&80089&$283^2$&1576&1.658\\
83521&$17^4$&1614&1.659&85849&$293^2$&1637&1.658&94249&$307^2$&1723&1.659\\
 96721&$311^2$&1748&1.659&97969&$313^2$&1761&1.660&100489&$317^2$&1786&1.661\\
 103823&$47^3$&1824&1.666&109561&$331^2$&1878&1.666&113569&$337^2$&1917&1.668\\
117649&$7^6$&1954&1.668&120409&$347^2$&1985&1.673&121801&$349^2$&1999&1.674\\
124609&$353^2$&2021&1.672&128881&$359^2$&2058&1.672&130321&$19^4$&2070&1.671\\
131072&$2^{17}$&2077&1.672&134689&$367^2$&2110&1.673&139129&$373^2$&2142&1.669\\
160801&$401^2$&2332&1.680&208849&$457^2$&2686&1.680&253009&$503^2$&2991&1.686\\
 \hline
 \end{tabular}
\newpage
\include{AC_TabPrime}

\end{document}

%% file: AC_TabPrime.tex
\textbf{Table 4.} The smallest known sizes $\overline{t}(q)$ of AC-subsets in $\mathrm{PG}(2, q)$ and values $\overline{t^*}(q)$,\\
$13\le q\le 33013$, $q$ prime \medskip

{\footnotesize\begin{tabular}
{@{}rr@{\,\,\,\,}c@{\,}|@{\,\,}rr@{\,\,\,\,}c@{\,}|@{\,\,}rr@{\,\,\,\,}c@{\,}|@{\,\,}rr@{\,\,\,\,}c@{\,}|@{\,\,}rr@{\,\,\,\,}c@{}}\hline
&&$\overline{t^*}(q)$&&&$\overline{t^*}(q)$&&&$\overline{t^*}(q)$&&&$\overline{t^*}(q)$&&&$\overline{t^*}(q) ^ {\vphantom{H^{L}}}$\\
$q$&$\overline{t}(q)$&$<$&$q$&$\overline{t}(q)$&$<$&$q$&$\overline{t}(q)$&$<$&$q$&$\overline{t}(q)$&$<$&$q$&$\overline{t}(q)$&$<$\\\hline
    13 &    8 & 1.386 &     17 &   10 & 1.441 &     19 &   11 & 1.471 &     23 &   12 &       &     29 &   13 &      \\
    31 &   14 &       &     37 &   16 &       &     41 &   16 &       &     43 &   16 &       &     47 &   18 &      \\
    53 &   20 &       &     59 &   21 &       &     61 &   22 &       &     67 &   23 &       &     71 &   24 &      \\
    73 &   24 &       &     79 &   25 &       &     83 &   26 &       &     89 &   24 &       &     97 &   29 &      \\
   101 &   30 &       &    103 &   30 &       &    107 &   31 &       &    109 &   32 &       &    113 &   32 &      \\
   127 &   35 &       &    131 &   36 &       &    137 &   37 &       &    139 &   37 &       &    149 &   38 &      \\
\hline
\end{tabular}}
\newpage

\textbf{Table 4.} Continue 1 \medskip

{\footnotesize\begin{tabular}
{@{}rr@{\,\,\,\,}c@{\,}|@{\,\,}rr@{\,\,\,\,}c@{\,}|@{\,\,}rr@{\,\,\,\,}c@{\,}|@{\,\,}rr@{\,\,\,\,}c@{\,}|@{\,\,}rr@{\,\,\,\,}c@{}}\hline
&&$\overline{t^*}(q)$&&&$\overline{t^*}(q)$&&&$\overline{t^*}(q)$&&&$\overline{t^*}(q)$&&&$\overline{t^*}(q) ^ {\vphantom{H^{L}}}$\\
$q$&$\overline{t}(q)$&$<$&$q$&$\overline{t}(q)$&$<$&$q$&$\overline{t}(q)$&$<$&$q$&$\overline{t}(q)$&$<$&$q$&$\overline{t}(q)$&$<$\\\hline
  1381 &  153 &       &   1399 &  155 &       &   1409 &  156 &       &   1423 &  157 &       &   1427 &  157 &      \\
  1429 &  157 &       &   1433 &  157 &       &   1439 &  158 &       &   1447 &  158 &       &   1451 &  159 & 1.548\\
  1453 &  159 &       &   1459 &  159 &       &   1471 &  160 &       &   1481 &  160 &       &   1483 &  161 & 1.548\\
  1487 &  161 &       &   1489 &  161 &       &   1493 &  161 &       &   1499 &  162 & 1.548 &   1511 &  162 &      \\
  1523 &  163 &       &   1531 &  163 &       &   1543 &  164 &       &   1549 &  165 &       &   1553 &  165 &      \\
  1559 &  166 & 1.551 &   1567 &  166 &       &   1571 &  166 &       &   1579 &  167 &       &   1583 &  167 &      \\
\hline
\end{tabular}}
\newpage

\textbf{Table 4.} Continue 2 \medskip

{\footnotesize\begin{tabular}
{@{}rr@{\,\,\,\,}c@{\,}|@{\,\,}rr@{\,\,\,\,}c@{\,}|@{\,\,}rr@{\,\,\,\,}c@{\,}|@{\,\,}rr@{\,\,\,\,}c@{\,}|@{\,\,}rr@{\,\,\,\,}c@{}}\hline
&&$\overline{t^*}(q)$&&&$\overline{t^*}(q)$&&&$\overline{t^*}(q)$&&&$\overline{t^*}(q)$&&&$\overline{t^*}(q) ^ {\vphantom{H^{L}}}$\\
$q$&$\overline{t}(q)$&$<$&$q$&$\overline{t}(q)$&$<$&$q$&$\overline{t}(q)$&$<$&$q$&$\overline{t}(q)$&$<$&$q$&$\overline{t}(q)$&$<$\\\hline
  3041 &  245 &       &   3049 &  246 &       &   3061 &  247 &       &   3067 &  248 &       &   3079 &  248 &      \\
  3083 &  249 & 1.583 &   3089 &  248 &       &   3109 &  249 &       &   3119 &  249 &       &   3121 &  250 &      \\
  3137 &  250 &       &   3163 &  251 &       &   3167 &  251 &       &   3169 &  252 &       &   3181 &  252 &      \\
  3187 &  253 &       &   3191 &  253 &       &   3203 &  253 &       &   3209 &  254 &       &   3217 &  255 &      \\
  3221 &  254 &       &   3229 &  255 &       &   3251 &  255 &       &   3253 &  255 &       &   3257 &  256 &      \\
  3259 &  256 &       &   3271 &  256 &       &   3299 &  258 &       &   3301 &  258 &       &   3307 &  257 &      \\
\hline
\end{tabular}}
\newpage

\textbf{Table 4.} Continue 3 \medskip

{\footnotesize\begin{tabular}
{@{}rr@{\,\,\,\,}c@{\,}|@{\,\,}rr@{\,\,\,\,}c@{\,}|@{\,\,}rr@{\,\,\,\,}c@{\,}|@{\,\,}rr@{\,\,\,\,}c@{\,}|@{\,\,}rr@{\,\,\,\,}c@{}}\hline
&&$\overline{t^*}(q)$&&&$\overline{t^*}(q)$&&&$\overline{t^*}(q)$&&&$\overline{t^*}(q)$&&&$\overline{t^*}(q) ^ {\vphantom{H^{L}}}$\\
$q$&$\overline{t}(q)$&$<$&$q$&$\overline{t}(q)$&$<$&$q$&$\overline{t}(q)$&$<$&$q$&$\overline{t}(q)$&$<$&$q$&$\overline{t}(q)$&$<$\\\hline
  4861 &  323 &       &   4871 &  323 &       &   4877 &  323 &       &   4889 &  323 &       &   4903 &  324 &      \\
  4909 &  325 & 1.592 &   4919 &  325 &       &   4931 &  325 &       &   4933 &  325 &       &   4937 &  325 &      \\
  4943 &  326 &       &   4951 &  325 &       &   4957 &  326 &       &   4967 &  327 &       &   4969 &  327 &      \\
  4973 &  327 &       &   4987 &  327 &       &   4993 &  328 &       &   4999 &  328 &       &   5003 &  328 &      \\
  5009 &  328 &       &   5011 &  329 & 1.593 &   5021 &  329 &       &   5023 &  329 &       &   5039 &  328 &      \\
  5051 &  330 &       &   5059 &  330 &       &   5077 &  331 &       &   5081 &  331 &       &   5087 &  331 &      \\
\hline
\end{tabular}}
\newpage

\textbf{Table 4.} Continue 4 \medskip

{\footnotesize\begin{tabular}
{@{}rr@{\,\,\,\,}c@{\,}|@{\,\,}rr@{\,\,\,\,}c@{\,}|@{\,\,}rr@{\,\,\,\,}c@{\,}|@{\,\,}rr@{\,\,\,\,}c@{\,}|@{\,\,}rr@{\,\,\,\,}c@{}}\hline
&&$\overline{t^*}(q)$&&&$\overline{t^*}(q)$&&&$\overline{t^*}(q)$&&&$\overline{t^*}(q)$&&&$\overline{t^*}(q) ^ {\vphantom{H^{L}}}$\\
$q$&$\overline{t}(q)$&$<$&$q$&$\overline{t}(q)$&$<$&$q$&$\overline{t}(q)$&$<$&$q$&$\overline{t}(q)$&$<$&$q$&$\overline{t}(q)$&$<$\\\hline
  6709 &  388 &       &   6719 &  389 &       &   6733 &  389 &       &   6737 &  389 &       &   6761 &  390 &      \\
  6763 &  390 &       &   6779 &  390 &       &   6781 &  391 &       &   6791 &  391 &       &   6793 &  391 &      \\
  6803 &  391 &       &   6823 &  392 &       &   6827 &  391 &       &   6829 &  392 &       &   6833 &  392 &      \\
  6841 &  392 &       &   6857 &  392 &       &   6863 &  393 &       &   6869 &  393 &       &   6871 &  393 &      \\
  6883 &  394 &       &   6899 &  395 &       &   6907 &  395 &       &   6911 &  395 &       &   6917 &  395 &      \\
  6947 &  396 &       &   6949 &  396 &       &   6959 &  397 &       &   6961 &  397 &       &   6967 &  397 &      \\
\hline
\end{tabular}}
\newpage

\textbf{Table 4.} Continue 5 \medskip

{\footnotesize\begin{tabular}
{@{}rr@{\,\,\,\,}c@{\,}|@{\,\,}rr@{\,\,\,\,}c@{\,}|@{\,\,}rr@{\,\,\,\,}c@{\,}|@{\,\,}rr@{\,\,\,\,}c@{\,}|@{\,\,}rr@{\,\,\,\,}c@{}}\hline
&&$\overline{t^*}(q)$&&&$\overline{t^*}(q)$&&&$\overline{t^*}(q)$&&&$\overline{t^*}(q)$&&&$\overline{t^*}(q) ^ {\vphantom{H^{L}}}$\\
$q$&$\overline{t}(q)$&$<$&$q$&$\overline{t}(q)$&$<$&$q$&$\overline{t}(q)$&$<$&$q$&$\overline{t}(q)$&$<$&$q$&$\overline{t}(q)$&$<$\\\hline
  8681 &  450 &       &   8689 &  451 &       &   8693 &  450 &       &   8699 &  450 &       &   8707 &  451 &      \\
  8713 &  450 &       &   8719 &  451 &       &   8731 &  451 &       &   8737 &  451 &       &   8741 &  452 &      \\
  8747 &  452 &       &   8753 &  452 &       &   8761 &  452 &       &   8779 &  452 &       &   8783 &  452 &      \\
  8803 &  453 &       &   8807 &  453 &       &   8819 &  454 &       &   8821 &  454 &       &   8831 &  455 &      \\
  8837 &  454 &       &   8839 &  455 &       &   8849 &  455 &       &   8861 &  455 &       &   8863 &  456 &      \\
  8867 &  456 &       &   8887 &  455 &       &   8893 &  456 &       &   8923 &  456 &       &   8929 &  456 &      \\
\hline
\end{tabular}}
\newpage

\textbf{Table 4.} Continue 6 \medskip

{\footnotesize\begin{tabular}
{@{}rr@{\,\,\,\,}c@{\,}|@{\,\,}rr@{\,\,\,\,}c@{\,}|@{\,\,}rr@{\,\,\,\,}c@{\,}|@{\,\,}rr@{\,\,\,\,}c@{\,}|@{\,\,}rr@{\,\,\,\,}c@{}}\hline
&&$\overline{t^*}(q)$&&&$\overline{t^*}(q)$&&&$\overline{t^*}(q)$&&&$\overline{t^*}(q)$&&&$\overline{t^*}(q) ^ {\vphantom{H^{L}}}$\\
$q$&$\overline{t}(q)$&$<$&$q$&$\overline{t}(q)$&$<$&$q$&$\overline{t}(q)$&$<$&$q$&$\overline{t}(q)$&$<$&$q$&$\overline{t}(q)$&$<$\\\hline
 10627 &  505 &       &  10631 &  505 &       &  10639 &  504 &       &  10651 &  505 &       &  10657 &  505 &      \\
 10663 &  507 &       &  10667 &  505 &       &  10687 &  506 &       &  10691 &  506 &       &  10709 &  507 &      \\
 10711 &  507 &       &  10723 &  507 &       &  10729 &  507 &       &  10733 &  508 &       &  10739 &  507 &      \\
 10753 &  508 &       &  10771 &  509 &       &  10781 &  509 &       &  10789 &  509 &       &  10799 &  510 &      \\
 10831 &  510 &       &  10837 &  511 &       &  10847 &  511 &       &  10853 &  511 &       &  10859 &  511 &      \\
 10861 &  511 &       &  10867 &  511 &       &  10883 &  511 &       &  10889 &  512 &       &  10891 &  511 &      \\
\hline
\end{tabular}}
\newpage

\textbf{Table 4.} Continue 7 \medskip

{\footnotesize\begin{tabular}
{@{}rr@{\,\,\,\,}c@{\,}|@{\,\,}rr@{\,\,\,\,}c@{\,}|@{\,\,}rr@{\,\,\,\,}c@{\,}|@{\,\,}rr@{\,\,\,\,}c@{\,}|@{\,\,}rr@{\,\,\,\,}c@{}}\hline
&&$\overline{t^*}(q)$&&&$\overline{t^*}(q)$&&&$\overline{t^*}(q)$&&&$\overline{t^*}(q)$&&&$\overline{t^*}(q) ^ {\vphantom{H^{L}}}$\\
$q$&$\overline{t}(q)$&$<$&$q$&$\overline{t}(q)$&$<$&$q$&$\overline{t}(q)$&$<$&$q$&$\overline{t}(q)$&$<$&$q$&$\overline{t}(q)$&$<$\\\hline
 12647 &  556 &       &  12653 &  558 & 1.615 &  12659 &  557 &       &  12671 &  557 &       &  12689 &  558 &      \\
 12697 &  558 &       &  12703 &  558 &       &  12713 &  559 &       &  12721 &  559 &       &  12739 &  559 &      \\
 12743 &  559 &       &  12757 &  559 &       &  12763 &  559 &       &  12781 &  561 &       &  12791 &  561 &      \\
 12799 &  560 &       &  12809 &  561 &       &  12821 &  561 &       &  12823 &  562 &       &  12829 &  561 &      \\
 12841 &  560 &       &  12853 &  562 &       &  12889 &  563 &       &  12893 &  563 &       &  12899 &  563 &      \\
 12907 &  564 &       &  12911 &  564 &       &  12917 &  563 &       &  12919 &  563 &       &  12923 &  564 &      \\
\hline
\end{tabular}}
\newpage

\textbf{Table 4.} Continue 8 \medskip

{\footnotesize\begin{tabular}
{@{}rr@{\,\,\,\,}c@{\,}|@{\,\,}rr@{\,\,\,\,}c@{\,}|@{\,\,}rr@{\,\,\,\,}c@{\,}|@{\,\,}rr@{\,\,\,\,}c@{\,}|@{\,\,}rr@{\,\,\,\,}c@{}}\hline
&&$\overline{t^*}(q)$&&&$\overline{t^*}(q)$&&&$\overline{t^*}(q)$&&&$\overline{t^*}(q)$&&&$\overline{t^*}(q) ^ {\vphantom{H^{L}}}$\\
$q$&$\overline{t}(q)$&$<$&$q$&$\overline{t}(q)$&$<$&$q$&$\overline{t}(q)$&$<$&$q$&$\overline{t}(q)$&$<$&$q$&$\overline{t}(q)$&$<$\\\hline
 14741 &  607 &       &  14747 &  608 &       &  14753 &  608 &       &  14759 &  608 &       &  14767 &  608 &      \\
 14771 &  607 &       &  14779 &  609 &       &  14783 &  609 &       &  14797 &  608 &       &  14813 &  609 &      \\
 14821 &  609 &       &  14827 &  609 &       &  14831 &  610 &       &  14843 &  610 &       &  14851 &  610 &      \\
 14867 &  610 &       &  14869 &  610 &       &  14879 &  610 &       &  14887 &  611 &       &  14891 &  611 &      \\
 14897 &  611 &       &  14923 &  611 &       &  14929 &  611 &       &  14939 &  613 &       &  14947 &  613 &      \\
 14951 &  612 &       &  14957 &  612 &       &  14969 &  612 &       &  14983 &  613 &       &  15013 &  614 &      \\
\hline
\end{tabular}}
\newpage

\textbf{Table 4.} Continue 9 \medskip

{\footnotesize\begin{tabular}
{@{}rr@{\,\,\,\,}c@{\,}|@{\,\,}rr@{\,\,\,\,}c@{\,}|@{\,\,}rr@{\,\,\,\,}c@{\,}|@{\,\,}rr@{\,\,\,\,}c@{\,}|@{\,\,}rr@{\,\,\,\,}c@{}}\hline
&&$\overline{t^*}(q)$&&&$\overline{t^*}(q)$&&&$\overline{t^*}(q)$&&&$\overline{t^*}(q)$&&&$\overline{t^*}(q) ^ {\vphantom{H^{L}}}$\\
$q$&$\overline{t}(q)$&$<$&$q$&$\overline{t}(q)$&$<$&$q$&$\overline{t}(q)$&$<$&$q$&$\overline{t}(q)$&$<$&$q$&$\overline{t}(q)$&$<$\\\hline
 16823 &  654 &       &  16829 &  654 &       &  16831 &  654 &       &  16843 &  655 &       &  16871 &  656 &      \\
 16879 &  656 &       &  16883 &  655 &       &  16889 &  655 &       &  16901 &  656 &       &  16903 &  657 &      \\
 16921 &  657 &       &  16927 &  657 &       &  16931 &  657 &       &  16937 &  657 &       &  16943 &  656 &      \\
 16963 &  657 &       &  16979 &  657 &       &  16981 &  658 &       &  16987 &  658 &       &  16993 &  657 &      \\
 17011 &  659 &       &  17021 &  659 &       &  17027 &  659 &       &  17029 &  658 &       &  17033 &  659 &      \\
 17041 &  659 &       &  17047 &  661 & 1.622 &  17053 &  660 &       &  17077 &  660 &       &  17093 &  661 &      \\
\hline
\end{tabular}}
\newpage

\textbf{Table 4.} Continue 10 \medskip

{\footnotesize\begin{tabular}
{@{}rr@{\,\,\,\,}c@{\,}|@{\,\,}rr@{\,\,\,\,}c@{\,}|@{\,\,}rr@{\,\,\,\,}c@{\,}|@{\,\,}rr@{\,\,\,\,}c@{\,}|@{\,\,}rr@{\,\,\,\,}c@{}}\hline
&&$\overline{t^*}(q)$&&&$\overline{t^*}(q)$&&&$\overline{t^*}(q)$&&&$\overline{t^*}(q)$&&&$\overline{t^*}(q) ^ {\vphantom{H^{L}}}$\\
$q$&$\overline{t}(q)$&$<$&$q$&$\overline{t}(q)$&$<$&$q$&$\overline{t}(q)$&$<$&$q$&$\overline{t}(q)$&$<$&$q$&$\overline{t}(q)$&$<$\\\hline
 18959 &  700 &       &  18973 &  700 &       &  18979 &  701 &       &  19001 &  702 &       &  19009 &  701 &      \\
 19013 &  701 &       &  19031 &  702 &       &  19037 &  702 &       &  19051 &  702 &       &  19069 &  702 &      \\
 19073 &  702 &       &  19079 &  703 &       &  19081 &  703 &       &  19087 &  703 &       &  19121 &  704 &      \\
 19139 &  704 &       &  19141 &  705 &       &  19157 &  704 &       &  19163 &  703 &       &  19181 &  706 & 1.624\\
 19183 &  704 &       &  19207 &  706 &       &  19211 &  705 &       &  19213 &  705 &       &  19219 &  706 &      \\
 19231 &  705 &       &  19237 &  705 &       &  19249 &  706 &       &  19259 &  707 &       &  19267 &  706 &      \\
\hline
\end{tabular}}
\newpage

\textbf{Table 4.} Continue 11 \medskip

{\footnotesize\begin{tabular}
{@{}rr@{\,\,\,\,}c@{\,}|@{\,\,}rr@{\,\,\,\,}c@{\,}|@{\,\,}rr@{\,\,\,\,}c@{\,}|@{\,\,}rr@{\,\,\,\,}c@{\,}|@{\,\,}rr@{\,\,\,\,}c@{}}\hline
&&$\overline{t^*}(q)$&&&$\overline{t^*}(q)$&&&$\overline{t^*}(q)$&&&$\overline{t^*}(q)$&&&$\overline{t^*}(q) ^ {\vphantom{H^{L}}}$\\
$q$&$\overline{t}(q)$&$<$&$q$&$\overline{t}(q)$&$<$&$q$&$\overline{t}(q)$&$<$&$q$&$\overline{t}(q)$&$<$&$q$&$\overline{t}(q)$&$<$\\\hline
 21089 &  745 &       &  21101 &  745 &       &  21107 &  743 &       &  21121 &  745 &       &  21139 &  745 &      \\
 21143 &  746 &       &  21149 &  743 &       &  21157 &  745 &       &  21163 &  745 &       &  21169 &  744 &      \\
 21179 &  745 &       &  21187 &  746 &       &  21191 &  745 &       &  21193 &  745 &       &  21211 &  746 &      \\
 21221 &  746 &       &  21227 &  746 &       &  21247 &  746 &       &  21269 &  747 &       &  21277 &  747 &      \\
 21283 &  748 &       &  21313 &  749 &       &  21317 &  748 &       &  21319 &  747 &       &  21323 &  748 &      \\
 21341 &  749 &       &  21347 &  748 &       &  21377 &  749 &       &  21379 &  749 &       &  21383 &  748 &      \\
\hline
\end{tabular}}
\newpage

\textbf{Table 4.} Continue 12 \medskip

{\footnotesize\begin{tabular}
{@{}rr@{\,\,\,\,}c@{\,}|@{\,\,}rr@{\,\,\,\,}c@{\,}|@{\,\,}rr@{\,\,\,\,}c@{\,}|@{\,\,}rr@{\,\,\,\,}c@{\,}|@{\,\,}rr@{\,\,\,\,}c@{}}\hline
&&$\overline{t^*}(q)$&&&$\overline{t^*}(q)$&&&$\overline{t^*}(q)$&&&$\overline{t^*}(q)$&&&$\overline{t^*}(q) ^ {\vphantom{H^{L}}}$\\
$q$&$\overline{t}(q)$&$<$&$q$&$\overline{t}(q)$&$<$&$q$&$\overline{t}(q)$&$<$&$q$&$\overline{t}(q)$&$<$&$q$&$\overline{t}(q)$&$<$\\\hline
 23173 &  784 &       &  23189 &  784 &       &  23197 &  785 &       &  23201 &  784 &       &  23203 &  784 &      \\
 23209 &  785 &       &  23227 &  785 &       &  23251 &  786 &       &  23269 &  786 &       &  23279 &  785 &      \\
 23291 &  786 &       &  23293 &  786 &       &  23297 &  786 &       &  23311 &  787 &       &  23321 &  787 &      \\
 23327 &  786 &       &  23333 &  787 &       &  23339 &  787 &       &  23357 &  788 &       &  23369 &  788 &      \\
 23371 &  788 &       &  23399 &  789 &       &  23417 &  788 &       &  23431 &  788 &       &  23447 &  789 &      \\
 23459 &  789 &       &  23473 &  789 &       &  23497 &  791 &       &  23509 &  790 &       &  23531 &  791 &      \\
\hline
\end{tabular}}
\newpage

\textbf{Table 4.} Continue 13 \medskip

{\footnotesize\begin{tabular}
{@{}rr@{\,\,\,\,}c@{\,}|@{\,\,}rr@{\,\,\,\,}c@{\,}|@{\,\,}rr@{\,\,\,\,}c@{\,}|@{\,\,}rr@{\,\,\,\,}c@{\,}|@{\,\,}rr@{\,\,\,\,}c@{}}\hline
&&$\overline{t^*}(q)$&&&$\overline{t^*}(q)$&&&$\overline{t^*}(q)$&&&$\overline{t^*}(q)$&&&$\overline{t^*}(q) ^ {\vphantom{H^{L}}}$\\
$q$&$\overline{t}(q)$&$<$&$q$&$\overline{t}(q)$&$<$&$q$&$\overline{t}(q)$&$<$&$q$&$\overline{t}(q)$&$<$&$q$&$\overline{t}(q)$&$<$\\\hline
 25409 &  825 &       &  25411 &  826 &       &  25423 &  827 &       &  25439 &  827 &       &  25447 &  826 &      \\
 25453 &  827 &       &  25457 &  826 &       &  25463 &  826 &       &  25469 &  827 &       &  25471 &  826 &      \\
 25523 &  829 &       &  25537 &  827 &       &  25541 &  829 &       &  25561 &  828 &       &  25577 &  828 &      \\
 25579 &  829 &       &  25583 &  829 &       &  25589 &  828 &       &  25601 &  830 &       &  25603 &  829 &      \\
 25609 &  830 &       &  25621 &  830 &       &  25633 &  831 & 1.630 &  25639 &  830 &       &  25643 &  830 &      \\
 25657 &  831 &       &  25667 &  830 &       &  25673 &  830 &       &  25679 &  830 &       &  25693 &  831 &      \\
\hline
\end{tabular}}
\newpage

\textbf{Table 4.} Continue 14 \medskip

{\footnotesize\begin{tabular}
{@{}rr@{\,\,\,\,}c@{\,}|@{\,\,}rr@{\,\,\,\,}c@{\,}|@{\,\,}rr@{\,\,\,\,}c@{\,}|@{\,\,}rr@{\,\,\,\,}c@{\,}|@{\,\,}rr@{\,\,\,\,}c@{}}\hline
&&$\overline{t^*}(q)$&&&$\overline{t^*}(q)$&&&$\overline{t^*}(q)$&&&$\overline{t^*}(q)$&&&$\overline{t^*}(q) ^ {\vphantom{H^{L}}}$\\
$q$&$\overline{t}(q)$&$<$&$q$&$\overline{t}(q)$&$<$&$q$&$\overline{t}(q)$&$<$&$q$&$\overline{t}(q)$&$<$&$q$&$\overline{t}(q)$&$<$\\\hline
 27647 &  864 &       &  27653 &  865 &       &  27673 &  867 &       &  27689 &  866 &       &  27691 &  866 &      \\
 27697 &  867 &       &  27701 &  866 &       &  27733 &  868 &       &  27737 &  867 &       &  27739 &  867 &      \\
 27743 &  867 &       &  27749 &  867 &       &  27751 &  867 &       &  27763 &  868 &       &  27767 &  868 &      \\
 27773 &  866 &       &  27779 &  869 &       &  27791 &  868 &       &  27793 &  869 &       &  27799 &  868 &      \\
 27803 &  869 &       &  27809 &  868 &       &  27817 &  868 &       &  27823 &  868 &       &  27827 &  868 &      \\
 27847 &  868 &       &  27851 &  869 &       &  27883 &  870 &       &  27893 &  869 &       &  27901 &  870 &      \\
\hline
\end{tabular}}
\newpage

\textbf{Table 4.} Continue 15 \medskip

{\footnotesize\begin{tabular}
{@{}rr@{\,\,\,\,}c@{\,}|@{\,\,}rr@{\,\,\,\,}c@{\,}|@{\,\,}rr@{\,\,\,\,}c@{\,}|@{\,\,}rr@{\,\,\,\,}c@{\,}|@{\,\,}rr@{\,\,\,\,}c@{}}\hline
&&$\overline{t^*}(q)$&&&$\overline{t^*}(q)$&&&$\overline{t^*}(q)$&&&$\overline{t^*}(q)$&&&$\overline{t^*}(q) ^ {\vphantom{H^{L}}}$\\
$q$&$\overline{t}(q)$&$<$&$q$&$\overline{t}(q)$&$<$&$q$&$\overline{t}(q)$&$<$&$q$&$\overline{t}(q)$&$<$&$q$&$\overline{t}(q)$&$<$\\\hline
 29833 &  903 &       &  29837 &  903 &       &  29851 &  904 &       &  29863 &  905 &       &  29867 &  903 &      \\
 29873 &  905 &       &  29879 &  904 &       &  29881 &  903 &       &  29917 &  904 &       &  29921 &  905 &      \\
 29927 &  904 &       &  29947 &  905 &       &  29959 &  906 &       &  29983 &  906 &       &  29989 &  907 &      \\
 30011 &  907 &       &  30013 &  906 &       &  30029 &  907 &       &  30047 &  906 &       &  30059 &  908 &      \\
 30071 &  907 &       &  30089 &  908 &       &  30091 &  907 &       &  30097 &  908 &       &  30103 &  909 &      \\
 30109 &  909 &       &  30113 &  908 &       &  30119 &  909 &       &  30133 &  908 &       &  30137 &  909 &      \\
\hline
\end{tabular}}
\newpage

\textbf{Table 4.} Continue 16 \medskip

{\footnotesize\begin{tabular}
{@{}rr@{\,\,\,\,}c@{\,}|@{\,\,}rr@{\,\,\,\,}c@{\,}|@{\,\,}rr@{\,\,\,\,}c@{\,}|@{\,\,}rr@{\,\,\,\,}c@{\,}|@{\,\,}rr@{\,\,\,\,}c@{}}\hline
&&$\overline{t^*}(q)$&&&$\overline{t^*}(q)$&&&$\overline{t^*}(q)$&&&$\overline{t^*}(q)$&&&$\overline{t^*}(q) ^ {\vphantom{H^{L}}}$\\
$q$&$\overline{t}(q)$&$<$&$q$&$\overline{t}(q)$&$<$&$q$&$\overline{t}(q)$&$<$&$q$&$\overline{t}(q)$&$<$&$q$&$\overline{t}(q)$&$<$\\\hline
 32117 &  942 &       &  32119 &  941 &       &  32141 &  941 &       &  32143 &  942 &       &  32159 &  943 &      \\
 32173 &  943 &       &  32183 &  943 &       &  32189 &  943 &       &  32191 &  941 &       &  32203 &  944 &      \\
 32213 &  944 &       &  32233 &  944 &       &  32237 &  943 &       &  32251 &  943 &       &  32257 &  945 &      \\
 32261 &  945 &       &  32297 &  945 &       &  32299 &  946 &       &  32303 &  946 &       &  32309 &  946 &      \\
 32321 &  945 &       &  32323 &  945 &       &  32327 &  945 &       &  32341 &  944 &       &  32353 &  946 &      \\
 32359 &  947 &       &  32363 &  946 &       &  32369 &  947 &       &  32371 &  947 &       &  32377 &  946 &      \\
\hline
\end{tabular}